\documentclass[preprint,12pt]{elsarticle}
\usepackage{amsthm,amsfonts,amssymb,amscd,amsmath,enumerate,verbatim,calc,graphicx,geometry}
\usepackage[all]{xy}
\newtheorem{theorem}{Theorem}[section]
\newtheorem{lemma}[theorem]{Lemma}
\newtheorem{proposition}[theorem]{Proposition}
\newtheorem{corollary}[theorem]{Corollary}
\theoremstyle{definition}
\theoremstyle{definitions}
\newtheorem{definition}[theorem]{Definition}

\newtheorem{remark}[theorem]{Remark}
\newtheorem{example}[theorem]{Example}
\theoremstyle{notations}

\theoremstyle{remarks}

\newcommand{\lo}{\longrightarrow}

\newcommand{\wt}{\widetilde}

\newcommand{\al}{\alpha}

\newcommand{\bt}{\beta}

\journal{ }
\begin{document}
\begin{frontmatter}
\title{On h-Fibrations}

\author[]{Mehdi~Tajik$^1$}
\ead{mm.tj@stu.um.ac.ir}
\author[]{Behrooz~Mashayekhy$^{1,}$\corref{cor1}}
\ead{bmashf@um.ac.ir}
\author[]{Ali~ Pakdaman$^2$ }
\ead{a.pakdaman@gu.ac.ir}

\address{$^1$Department of Pure Mathematics, Center of Excellence in Analysis on Algebraic Structures,\\ Ferdowsi University of Mashhad,
P.O.Box 1159-91775, Mashhad, Iran.}
\address{$^2$Department of Mathematics, Faculty of Science, University of Golestan,\\
P.O.Box 155, Gorgan, Iran.}
\cortext[cor1]{Corresponding author}
\begin{abstract}
In this paper, we study h-fibrations, a weak homotopical version of fibrations which have weak covering homotopy property. We present some homotopical analogue of the notions related to fibrations and characterize h-fibrations using them. Then we construct some new categories by h-fibrations and deduce some results in these categories such as the existence of products and coproducts.
\end{abstract}

\begin{keyword}
fiber homotopy\sep weak covering homotopy property\sep h-fibration\sep homotopically path lifting.
\MSC[2010]{55R65, 55R05, 55R35, 57M12.}
\end{keyword}
\end{frontmatter}
\section{Introduction}
\subsection{Motivation}
A map $p:E\to B$ is said to be a fibration (Hurewicz fibration) if it has covering homotopy property with respect to every space, that is, for every space $X$, every map $\wt f:X\to E$ and every homotopy $F:X\times I\to B$ with $p\circ \widetilde{f}=F\circ J_0$, there exists a homtopy $\widetilde{F}:X\times I\to E$ such that $p\circ \widetilde{F}=F$ and $\widetilde{F}\circ J_0=\widetilde{f}$, where $J_0:X\to X\times I$ is $J_0(x)=(x,0)$.

Covering homotopy property is not invariant under fiber homotopy equivalence and hence any fiber homotopic map to a fibration is not necessarily a fibration. E. Fadell \cite{FA} introduced a new type of fibrations which do not have this defect.
 Also, Dold \cite{DO} considered a weak version of covering homotopy property introduced by Fuchs \cite{FU} which enjoys useful property of covering homotopy property such as exact homotopy sequence and spectral sequence and also, it is invariant under fiber homotopy equivalence.

 A fiber homotopy is a kind of homotopy which preserves points in their fibers during the homotopy (\cite{H,SP}) and weak covering homotopy property is obtained by replacing $\widetilde{F}\circ J_0=\widetilde{f}$ by $\widetilde{F}\circ J_0\simeq_p\widetilde{f}$ in the definition of covering homotopy property, (see \cite{DO, FU, P, PP}). Dold proved that under a weak local contractibility condition for a space $B$, a map $p:E\to B$ has weak covering homotopy property if and only if it is locally fiber homotopically trivial \cite[Theorem 6.4]{DO}. A map $p:E\to B$ is called h-fibration or Dold fibration if it has weak covering homotopy property with respect to every space.  A good characterization of fibrations and h-fibrations can be found in \cite{DO,PP,SP}.

A map $p:E\lo B$ is said to have unique path lifting property (upl) if for given paths $\alpha$ and $\alpha'$ in $E$ such that $p\circ \alpha=p\circ \alpha'$ and $\alpha(0)=\alpha'(0)$, we have $\alpha=\alpha'$ (see \cite{SP}). Unique path lifting property has an important role for fibrations, because make them very close to covering projections and also implies lifting theorem \cite[Theorem 2.4.5]{SP}. In \cite{MPT}, the authors presented a homotopical version of unique path lifting property and studied it's properties for fibrations.

 Here, after rehabilitating the definition of fiber homotopic maps with respect to arbitrary maps (instead of fibrations), we study h-fibrations with weakly unique path homotopically lifting property and give a sufficient condition which makes an h-fibration to be a fibration. We prove that an h-fibration has the homotopically path lifting property, which is a homotopical version of path lifting property. Also, we show that an h-fibration has homotopically lifting function. In Section 3, by proving that the composition of h-fibrations is an h-fibration, we introduce some new categories by h-fibrations $\mathrm{hFib}$, $\mathrm{hFibu}$ and $\mathrm{hFibwu}$. Then we compare them by the categories constructed by fibrations $\mathrm{Fib}$, $\mathrm{Fibu}$ and $\mathrm{Fibwu}$ (see \cite{MPT, SP}). Moreover, we show that these new categories have products and coproducts by introducing them.
\subsection{Preliminaries}
Throughout this paper, all spaces are path connected, unless otherwise stated. A \emph{map} $f:X\lo Y$ means a continuous function. A map $\al:I\lo X$ is called a path from $x_0=\al(0)$ to $x_1=\al(1)$ and it's inverse $\al^{-1}$ is defined by $\al^{-1}(t)=\al(1-t)$. For two paths $\al, \bt:I\lo X$ with $\al(1)=\bt(0)$, $\al*\bt$ denotes the usual concatenation of the two paths.\ Also, all homotopies between paths are assumed to be relative to end points.

\ For given maps $p:E\to B$ and $f:X\to B$, a map $\widetilde{f}:X\to E$ is  called a lift of $f$ if $p\circ \widetilde{f}=f$.
When $F:X\times I\longrightarrow Y$ is a map, we say that $F$ is a homotopy from $F_0$ to $F_1$ and write $F:F_0\simeq F_1$, where $F_i:X\lo Y$ is $F_i(x)=F(x,i)$, for $i=0,1$. The constant map from $X$ to $Y$ which sends all the points to $y\in Y$ is denoted by $C_y$.

For a toplological space $Y$, $Y^I$ is the space of paths in $Y$ and for a given map $f:X\to Y$, $P_f$ is the mapping path space, that is, $P_f={\{(x,\alpha)\in X\times Y^I|f(x)=\alpha(0)}\}$. Also, $p:P_f\to X$ by $p(x,\alpha)=x$ is a fibration which is called the mapping path fibration (see \cite{SP}).


\section{h-Fibrations}

For the definition of fiber homotopic maps with respect to a fibration and the definition of fiber homotopy equivalent fibrations, see \cite{SP}.\ We give here similar definitions for an arbitrary map and a few basic results that we need in sequel.
\begin{definition}
  Let $p:E\to B$ be a map.\ Two maps $f_0,f_1:X\to E$ are said to be fiber homotopic with respect to $p$, denoted by $f_0\simeq_p f_1$\ if there is a homotopy $F:f_0\simeq f_1$\ such that $p\circ F(x,t) = p\circ f_0(x)=p\circ f_1(x)$\ for every $x\in X$ and any $t\in I$.
\end{definition}
 We recall that for given maps $p_1:E_1\to B$ and $p_2:E_2\to B$, a map $f:E_1\to E_2$ is called fiber-preserving if $p_1=p_2\circ f$.
\begin{definition}
Two maps $p_1:E_1\to B$\ and $p_2:E_2\to B$ are said to be fiber homotopy equivalent, if there exist fiber preserving maps $f:E_1\to E_2$\ and $g:E_2\to E_1$ such that $g\circ f\simeq_{p_1} 1_{E_1}$\ and $f\circ g\simeq_{p_2} 1_{E_2}$.\ Each of the maps $f$ and $g$ is called a fiber homotopy equivalence.
\end{definition}
We have the following proposition for the fiber homotopy property.
 \begin{proposition}\label{P2.2}
\ Let $p:E\to B$ be a map.
\\(i) If $p':E'\to E$ and $f_0,f_1:X\to E'$ are maps such that $f_0\simeq_{p'}f_1$, then $f_0\simeq_{p\circ p'}f_1$.
\\ (ii) The fiber homotopy with respect to $p$ is an equivalence relation on the set of maps from $X$ to $E$.
\\ (iii) If $f,g:X\to E$ and $h:Z\to X$ are maps such that $f\simeq_{p}g$, then $f\circ h\simeq_{p}g\circ h$.
\\ (iv) If $f, g:X\to E'$ and $p':E'\to E$ are maps such that $f\simeq_{p\circ p'}g$, then $p'\circ f\simeq_p p'\circ g$.
\begin{proof}\
 \\ (i) Let $H:f_0\simeq_{p'}f_1$.
 Since $p'\circ H(x,t)=p'\circ f_0(x)=p'\circ f_1(x),$ we have
 $$(p\circ p')\circ H(x,t)=(p\circ p')\circ f_0(x)=(p\circ p')\circ f_1(x),$$ which implies that $H:f_0\simeq_{p\circ p'}f_1$.
\\ (ii) It is similar to the proof of ordinary homotopy relation.
\\ (iii) Let $H:X\times I\to E$ be a fiber homotopy from $f$ to $g$ with respect to $p$. Define $H':Z\times I\to E$ by $H'(z,t)=H(h(z),t)$.\ Then $H'$ is a homotopy from $f\circ h$ to $g\circ h$ and $$p\circ H'(z,t)=p\circ H(h(z),t)=p\circ f(h(z))=p\circ g(h(z)).$$
\\ (iv) Let $H:X\times I\to E'$ be the fiber homotopy $H:f\simeq_{p\circ p'}g$ and define $H':X\times I\to E$ by $H'(x,t)=p'\circ H(x,t)$. Then $p\circ H'(x,t)=p\circ p'\circ H(x,t)=p\circ p'\circ f(x)=p\circ p'\circ g(x)$ and so $H':p'\circ f\simeq_p p'\circ g$.
\end{proof}
 \end{proposition}
   A map $p:E\to B$ has weak covering homotopy property, abbreviated by wchp, if for every space X and any maps $\widetilde{f}:X\to E$, $F:X\times I\to B$ with $p\circ \widetilde{f}=F\circ J_0$, there exists a homotopy $\widetilde{F}:X\times I\to E$ such taht $p\circ \widetilde{F}=F$\ and $\widetilde{F}\circ J_0\simeq_p \widetilde{f}$. In Fact, Dold's definition was a bit different \cite{DO}. A map $p:E\to B$ is called an h-fibration if it has wchp \cite{PP}. By \cite[Proposition 5.2]{DO}, if a map is fiber homotopy equivalent to a fibration, it has wchp and so is an h-fibration. Also, in \cite{P} it is mentioned that every h-fibration is fiber homotopy equivalent to a fibration (see \cite[Proposition 1.15]{PP}). Since one can not find a detailed proof for this fact, we are going to give a proof for it.
First, for $f:X\rightarrow Y$, let $P_f$ be the mapping path space, $p:P_f\to X$ be the mapping path fibration and $h:X\to P_f$ be the map $h(x)=(x,C_{f(x)})$. Then $h$ is a section of $p$, moreover, $h$ and $p$ are homotopy inverse (\cite[Theorem 2.8.9]{SP}).
\begin{theorem}
  A map is an h-fibration if and only if it is fiber homotopy equivalent to a fibration.
    \begin{proof}
      Let $f:X\to Y$ be fiber homotopy equivalent to a fibration $f':X'\to Y$.\ There exist two maps $s:X\to X'$ and $s':X'\to X$ such that $s\circ s'\simeq_{f'}1_{X'}$, $s'\circ s\simeq_{f}1_{X}$, $f\circ s'=f'$ and $f'\circ s=f$ .\ If $\widetilde{f}:Z\to X$ and $F:Z\times I\to Y$ are maps such that $f\circ \widetilde{f}=F\circ J_0$, then $f'\circ s\circ \widetilde{f}=F\circ J_0$.\ Therefore by assumption, there is a homotopy $\widetilde{F}:Z\times I\to X'$ such that $f'\circ \widetilde{F}=F$ and $\widetilde{F}\circ J_0=s\circ \widetilde{f}$.\ Let $\widetilde{G}=s'\circ \widetilde{F}$.\ Then $$f\circ \widetilde{G}=f\circ s'\circ \widetilde{F}=f'\circ \widetilde{F}=F,$$ and by Proposition \ref{P2.2} $$\widetilde{G}\circ J_0=s'\circ \widetilde{F}\circ J_0=s'\circ s\circ \widetilde{f}\simeq_{f}1_{X}\circ \widetilde{f}=\widetilde{f}.$$\ Hence $f$ is an h-fibration.
      \\ Conversely, let $f:X\to Y$ be an h-fibration. Define a map $p_1:P_f\to Y$ by $p_1(x,\alpha)=\alpha(1)$ which is a fibration (see \cite[Theorem 2.8.9]{SP}). We show that $f$ and $p_1$ are fiber homotopy equivalent.\ Let $\gamma:P_f\times I\to Y$ be defined by $\gamma(x,\alpha,t)=\alpha(t)$ .\ Note that $$f\circ p(x,\alpha)=f(x)=\alpha(0)=\gamma(x,\alpha,0)=\gamma\circ J_0(x,\alpha),$$ where $J_0:P_f\to P_f\times I$ is the map $J_0(x,\alpha)=(x,\alpha,0)$.\ Since $f$ is an h-fibration, there exist homotopies $\widetilde{\gamma}:P_f\times I\to X$ and $T:P_f\times I\to X$ such that $f\circ \widetilde{\gamma}=\gamma$ and $T:\widetilde{\gamma}\circ J_0\simeq_f p$. Let $g:P_f\to X$ be defined by $g(x,\alpha)=\widetilde{\gamma}(x,\alpha,1)$. It is sufficient to show that $g\circ h\simeq_f 1_X$ and $h\circ g\simeq_{p_1} 1_{P_f}$. We have that $p_1\circ h=f$ and also $f\circ g=p_1$, because $f\circ g(x,\alpha)=f\circ \widetilde{\gamma}(x,\alpha,1)=\gamma(x,\alpha,1)=\alpha(1)=p_1(x,\alpha)$. Let $F:X\times I\to X$ be the map $F(x,t)=\widetilde{\gamma}(x,C_{f(x)},t)$.
      Since $F:F_0\simeq F_1$ and
      $$f\circ F(x,t)=f\circ \widetilde{\gamma}(x,C_{f(x)},t)=\gamma(x,C_{f(x)},t)=C_{f(x)}(t)=f(x),$$
      $$f\circ F_0(x)=f\circ F(x,0)=f\circ \widetilde{\gamma}(x,C_{f(x)},0)=\gamma(x,C_{f(x)},0)=C_{f(x)}(0)=f(x),$$
       $$f\circ F_1(x)=f\circ F(x,1)=f\circ \widetilde{\gamma}(x,C_{f(x)},1)=\gamma(x,C_{f(x)},1)=C_{f(x)}(1)=f(x),$$ $$F_1(x)=F(x,1)=\widetilde{\gamma}(x,C_{f(x)},1)=g(x,C_{f(x)})=g\circ h(x),$$ we have $F:F_0\simeq_f g\circ h$.\\
   \ Define $T':X\times I\to X$ by $T'(x,t)=T(x,C_{f(x)},t)$.\ Then $T':F_0\simeq_f 1_X$ since
    $$T'(x,0)=T(x,C_{f(x)},0)=\widetilde{\gamma}\circ J_0(x,C_{f(x)})=F(x,0)=F_0(x),$$
    $$T'(x,1)=T(x,C_{f(x)},1)=p(x,C_{f(x)})=x=1_X(x),$$
    $$f\circ T'(x,t)=f\circ T(x,C_{f(x)},t)=f\circ \widetilde{\gamma}\circ J_0(x,C_{f(x)})=f\circ F_0(x),$$
    $$f\circ T'(x,t)=f\circ T(x,C_{f(x)},t)=f\circ p(x,C_{f(x)})=f(x)=f\circ 1_X(x).$$
      Hence, transitivity of fiber homotopy implies that $g\circ h\simeq_f 1_X$.
    \\ For the second fiber homotopy, define $H:P_f\times I\to P_f$ by $H(x,\alpha,s)=(\widetilde{\gamma}(x,\alpha,s),\alpha_s)$, for in which $\alpha_s$ is the path $\alpha_s(t)=\alpha(s+t-st)$, for every $s,t\in I$.\ Clearly, $H:H_0\simeq H_1$.
     \ Moreover,
     $$p_1\circ H(x,\alpha,s)=p_1(\widetilde{\gamma}(x,\alpha,s),\alpha_s)=\alpha_s(1)=\alpha(1),$$
     $$p_1\circ H_0(x,\alpha)=p_1\circ H(x,\alpha,0)=p_1(\widetilde{\gamma}(x,\alpha,0),\alpha_0)=\alpha_0(1)=\alpha(1),$$
     $$p_1\circ H_1(x,\alpha)=p_1\circ H(x,\alpha,1)=p_1(\widetilde{\gamma}(x,\alpha,1),\alpha_1)=p_1(\widetilde{\gamma}(x,\alpha,1),C_{\alpha(1)})=C_{\alpha(1)}(1)=\alpha(1),$$ and hence $H:H_0\simeq_{p_1} H_1.$
    \ On the other hand, $$h\circ g(x,\alpha)=h(\widetilde{\gamma}(x,\alpha,1))=(\widetilde{\gamma}(x,\alpha,1),C_{f\circ \widetilde{\gamma}(x,\alpha,1)})=$$$$(\widetilde{\gamma}(x,\alpha,1),C_{\gamma(x,\alpha,1)})=
    (\widetilde{\gamma}(x,\alpha,1),C_{\alpha(1)})=H_1(x,\alpha).$$ Thus $h\circ g=H_1$ and so $H:H_0\simeq_{p_1}h\circ g$.\ Now, define $T'':P_f\times I\to P_f$ by $T''(x,\alpha,s)=(T(x,\alpha,s),\alpha)$.\ Note that $T''$ is well-define because $$f\circ T(x,\alpha,s)=f\circ p(x,\alpha)=f(x)=\alpha(0).$$
    \ Moreover, $T'':H_0\simeq_{p_1} 1_{P_f}$ since
     $$T''(x,\alpha,0)=(T(x,\alpha,0),\alpha)=(\widetilde{\gamma}(x,\alpha,0),\alpha)=H_0(x,\alpha),$$ $$T''(x,\alpha,1)=(T(x,\alpha,1),\alpha)=(p(x,\alpha),\alpha)=(x,\alpha)=1_{P_f}(x,\alpha),$$
     $$p_1\circ T''(x,\alpha,s)=p_1(T(x,\alpha,s),\alpha)=\alpha(1),$$
     $$p_1\circ H_0(x,\alpha)=p_1\circ H(x,\alpha,0)=p_1(\widetilde{\gamma}(x,\alpha,0),\alpha)=\alpha(1),$$
     $$p_1\circ 1_{P_f}(x,\alpha)=p_1(x,\alpha)=\alpha(1).$$
         Therefore, using $H$ and $T''$, we have
     $h\circ g\simeq_{p_1}1_{P_f}$.
     \end{proof}
\end{theorem}
For fibrations with path connected base space, any two fibers have the same homotopy type.\ Since, every h-fibration is fiber homotopy equivalent to a fibration, and a fiber homotopy equivalence can be thought as a family of homotopy
equivalences between corresponding fibers (\cite[Page 406]{H}), hence we have another proof for the following proposition.
\begin{proposition} \label{P2.5} (\cite[Proposition 1.12]{PP}).
  The fibers of an h-fibration have the same homotopy type.
\end{proposition}
\begin{corollary}
  If an h-fibration has the path connected base space with a path connected fiber, then its total space is also path connected.
  \begin{proof}
    Let $p:E\to B$ be an h-fibration with a path connected fiber. Then by Proposition \ref{P2.5}, every fiber of $p$ is path connected.
     Let $p':E'\to B$ be a fibration which is fiber homotopy equivalent to $p$. Then the fibers of $p'$ are path connected and so by
     \cite[Exercise 2.8.E.2]{SP}, $E'$ is path connected. By definition, there exist fiber preserving maps $f:E\to E'$ and $g:E'\to E$. If $x,y\in E$, then there exists a path $\alpha$ in $E'$ from $f(x)$ to $f(y)$. Since $g\circ f\simeq_p 1_E$, we have $x, g(f(x))\in p^{-1}(x)$, also $y, g(f(y))\in p^{-1}(y)$. Let $\beta$ be a path in $p^{-1}(x)$ from $x$ to $g(f(x))$ and $\gamma$ be a path in $p^{-1}(y)$ from $g(f(y))$ to $y$.\ Therefore $\beta\ast (g\circ \alpha)\ast \gamma$ is a path in $E$ from $x$ to $y$.
  \end{proof}
  \end{corollary}
 By definitions, every fibration is an h-fibration.\ But an h-fibration is not necessarily a fibration (for an example see \cite{DO}).\ In order to find a sufficient condition which makes an h-fibration a fibration, first consider the following lemma.
\begin{lemma}\label{L2.5}
  If a map $p:E\to B$ has upl and $f_0,f_1:X\to E$ are fiber homotopic with respect to $p$, then $f_0=f_1$.
  \begin{proof}
    Let $H:f_0\simeq_p f_1$.\ Then for every $x\in X$ and every $t\in I$, $p\circ H(x,t)=p\circ f_0(x)=p\circ f_1(x)$.\ For a fix $x\in X$, $H(x,-)$ is a path in the fiber $p^{-1}(p\circ f_0(x))$ and so $p\circ H(x,-)=C_{p\circ f_0(x)}$.\ Since $p\circ C_{f_0(x)}=C_{p\circ f_0(x)}$,
    $H(x,0)=f_0(x)=C_{f_0(x)}(0)$ and $p$ has upl, we have $H(x,-)=C_{f_0(x)}$.\ Hence $f_0(x)=H(x,0)=H(x,1)=f_1(x)$, as desired.
    \end{proof}
\end{lemma}
\begin{theorem}\label{T2.10}
  Every h-fibration with upl is a fibration.
\begin{proof}
Let $p:E\to B$ be an h-fibration. Also let $X$ be a topological space, $\widetilde{f}:X\to E$ and $F:X\times I\to B$ be maps such that $p\circ \widetilde{f}=F\circ J_0$.\ Then, there exists a homotopy $\widetilde{F}:X\times I\to E$ such that $p\circ \widetilde{F}=F$\ and $\widetilde{F}\circ J_0\simeq_p \widetilde{f}$.\ By Lemma \ref{L2.5}, $\widetilde{F}\circ J_0=\widetilde{f}$ which implies that $p$ is a fibration.
\end{proof}
\end{theorem}
Unique path lifting property is important in the study of fibrations because make them a covering map, when the base space is locally nice, i.e, locally path connected and semi-locally simply connected \cite{SP}. The authors introduced a homotopical version of upl in \cite{MPT} and studied its role in fibrations. Since here we are working with a weak homotopical version of fibrations, we are going to study h-fibrations with the homotopical version of upl.

A map $p:E\to B$ is said to have \emph{weakly unique path homotopically lifting property} abbreviated by wuphl, if by given two paths $\alpha$ and $\beta$ in $E$ with $\alpha(0)=\beta(0),\ \alpha(1)=\beta(1)$ and $p\circ\alpha \simeq  p\circ\beta,\ \mathrm{rel}\ \dot{I}$, then it follows that $\alpha \simeq \beta,\ \mathrm{rel}\ \dot{I}$ (see \cite{MPT}).
\\ The unique path lifting property for fibrations is equivalent to the fact that every path in any fiber is constant \cite[Theorem 2.2.5]{SP}.\ Also, the weakly unique path homotopically lifting property for fibrations is equivalent to the fact that every loop in any fiber is nullhomotopic \cite[Theorem 3.4]{MPT}.\ In the following, we show that these facts hold for h-fibrations.
\begin{proposition}\label{P2.14}
   An h-fibration $p:E\to B$ has upl if and only if every path in any fiber is constant.
   \begin{proof}
If $p$ has upl, then it is easy to see that every path in any fiber is constant. For the converse, let $\alpha, \beta:I\to E$ be two lifts of a path in $B$ started from the same point $\alpha(0)=\beta(0)$. Let $t\in I$ and consider the path $\gamma_t$ in $E$ from $\alpha(t)$ to $\beta(t)$ by
\begin{align*}
\gamma_t(t')=
\begin{cases}
  \alpha((1-2t')t), & t'\in [0,\frac 1 2] \\
  \beta((2t'-1)t), & t'\in [\frac 1 2, 1].
\end{cases}
\end{align*}
By assumption $p\circ \alpha=p\circ \beta$, then there exists a homotopy $F:p\circ \gamma_t\simeq C_{p\circ \alpha(t)}$, rel\ $\dot I$. Since $p$ has wchp, there exist homotopies $\wt F,H:I\times I\to E$ with $p\circ \wt F=F$ and $H:\wt F\circ J_0\simeq_p \gamma_t$. Thus
  $$p\circ \wt F(0,t)=F(0,t)=C_{p\circ \alpha(t)}(0)=p\circ \alpha(t),$$ $$p\circ \wt F(t,1)=F(t,1)=C_{p\circ \alpha(t)}(t)=p\circ \alpha(t),$$ and $$p\circ \wt F(1,t)=F(1,t)=C_{p\circ \alpha(t)}(1)=p\circ \alpha(t).$$ Therefore, $\wt F(0,-)\ast \wt F(-,1)\ast (\wt F(1,-))^{-1}$ is a path in the fiber $p^{-1}(p\circ \alpha(t))$. So by assumption it is constant, which implies that $\wt F(0,0)=\wt F(1,0)$. Now, note that $H(0,-)$ is a path from $\wt F(0,0)$ to $\gamma_t(0)=\alpha(t)$ in the fiber $p^{-1}(p\circ \gamma_t(0))$, and also $H(1,-)$ is a path from $\wt F(1,0)$ to $\gamma_t(1)=\beta(t)$ in the fiber $p^{-1}(p\circ \gamma_t(1))$. Then, since $\wt F(0,0)=\wt F(1,0)$ and $$p^{-1}(p\circ \gamma_t(0))=p^{-1}(p\circ \alpha(t))=p^{-1}(p\circ \beta(t))=p^{-1}(p\circ \gamma_t(1)),$$ there exists a path in this fiber from $\alpha(t)$ to $\beta(t)$, which by assumption it must be constant. Then, $\alpha(t)=\beta(t)$ and since $t$ is arbitrary we will have $\alpha=\beta$.
\end{proof}
\end{proposition}
\begin{proposition}\label{P2.6}\
  An h-fibration $p:E\to B$ has wuphl if and only if every loop in any fiber is nullhomotopic.
\end{proposition}
\begin{proof}
   Necessity is trivial.\ For the sufficiency, let $\wt{\alpha},\wt{\beta}:I\lo E$ be two paths with $\wt{\alpha}(0)=\wt{\beta}(0)$, $\wt{\alpha}(1)=\wt{\beta}(1)$ and $p\circ\wt{\al}\simeq p\circ\wt{\bt}$,\ rel\ $\dot{I}$.\ Let $\gamma:=\wt{\alpha}*\wt{\beta}^{-1}$ which is a loop at $\wt{\alpha}(0)$.\ Put $\tilde{x}_0=\wt{\al}(0)$ and $x_0=p(\tilde{x}_0)$, then we have $$p\circ\gamma=p\circ(\wt{\alpha}*\wt{\beta}^{-1})=(p\circ\wt{\alpha})*(p\circ\wt{\beta}^{-1})=(p\circ\wt{\alpha})*(p\circ\wt{\beta})^{-1}\simeq C_{x_0},\ \mathrm{rel}\ \dot{I}.$$
Let $F:p\circ\gamma\simeq C_{x_0}$,\ rel\ $\dot{I}$.\ Since $p$ is an h-fibration, there exist homotopies $\wt{F}, H:I\times I\lo E$ such that $p\circ\wt{F}=F$ and $H:\gamma\simeq_p\widetilde{F}\circ J_0$.\ Let $f:=\wt{F}(0,-),\ f':=H(0,-),\ g:=\wt{F}(- ,1),\ h:=\wt{F}(1,- )$ and $h':=H(1,-)$ which are paths in $E$ with $f'(1)=f(0),\ f(1)=g(0)$,\ $g(1)=h(1)=h^{-1}(0)$ and $h^{-1}(1)=h'^{-1}(0)$, so we can define $\eta:=f*g*h^{-1}$ and $\delta:=f'\ast \eta\ast h'^{-1}$.\ Note that $\delta$ is a closed path because
  $$\delta(0)=f'(0)=H(0,0)=\gamma(0)=\wt{\alpha}(0)=\wt{\beta}(0)=$$
   $$\wt{\beta}^{-1}(1)=\gamma(1)=H(1,0)=h'(0)=h'^{-1}(1)=\delta(1).$$
   Also, since
\begin{align*}
 & p\circ f'(t)=p\circ H(0,t)=p\circ \widetilde{F}\circ J_0(0)=p\circ \widetilde{F}(0,0)=F(0,0)=x_0=C_{x_0}(t),\quad\\
  & p\circ\eta=(p\circ f)*(p\circ g)*(p\circ h)^{-1}=F(0,-)*F(-,1)*(F(1,-))^{-1}=\quad \\
  &C_{x_0}*C_{x_0}*(C_{x_0})^{-1}=C_{x_0},\quad\\
  &p\circ h'(t)=p\circ H(1,t)=p\circ \widetilde{F}\circ J_0(1)=p\circ \widetilde{F}(1,0)=F(1,0)=x_0=C_{x_0}(t),\quad \\
\end{align*}
we have $p\circ \delta=(p\circ f')\ast(p\circ \eta)\ast(p\circ h'^{-1})=C_{x_0}$.\ Hence $\delta$ belongs to the fiber $p^{-1}(x_0)$ and so by assumption, $\delta$ is null.\ On the other hand, by definitions of $\gamma, \delta$ and $H, \widetilde{F}$ $\gamma\ast \delta\simeq C_{\tilde{x}_{0}}$, rel\ $\dot{I}$, which implies $\gamma\simeq C_{\tilde{x}_{0}}$, rel\ $\dot{I}$ and so $\wt{\alpha}\simeq\wt{\beta}$,\ rel\ $\dot{I}$.
\end{proof}
Obviously, if every loop in fibers of an h-fibration $p$ is constant, then $p$ has wuphl, but the converse is not necessarily true.\ For example, the h-fibration $pr_1:X\times Y\to X$, when $Y$ is any non-singleton simply connected space, has wuphl and also has nonconstant paths in its fibers. Since the fibers of two fiber homotopy equivalence fibrations (h-fibrations) have the same homotopy type, by Proposition \ref{P2.6} we have the following result.
\begin{corollary}\label{C2.12}\
  \\ (i) If two h-fibrations are fiber homotoy equivalent and one of them has wuphl, then so has the other one.
  \\ (ii) If an h-fibration is fiber homotoy equivalent to a fibration and one of them has wuphl, then so has the other one.

\end{corollary}
          A map $p:E\to B$ has path lifting property if for a given path $\alpha:I\to B$ with $\alpha(0)\in p(E)$ and every $e\in p^{-1}(\alpha (0))$ there exists a path $\widetilde{\alpha}$ in $E$ started at $e$, such that $p\circ \widetilde{\alpha}=\alpha$.
We know that every fibration has the path lifting property and the following example shows that an h-fibration does not necessarily have the path lifting property.
\begin{example}\label{ex1}
 Let $E=([-1,0]\times [-1,0])\cup ([0,1]\times [0,1])$, $B=[-1,1]$ and $p$ be the projection on the first component.\ Then $p$ is an h-fibration because for given maps $F:X\times I\to B$ and $\widetilde{f}:X\to E$ with $p\circ \widetilde{f}=F\circ J_0$ it suffices to define $\wt{F}:X\times I\to E$ by $\wt{F}(x,t)=(F(x,t),0)$. But, there is no lift for the path $\alpha(t)=t$, started from $(0,-\dfrac{1}{2})$.
\end{example}
In the following we give a homotopical analogue of path lifting property and show that h-fibrations enjoy this property.
\begin{definition}
   A map $p:E\to B$ has homotopically path lifting property if for a given $b\in B$, $e\in p^{-1}(b)$ and a path $\alpha$ in $B$ beginning at $b$, there exists a path $\wt\alpha$ in $E$ such that $\widetilde{\alpha}(0)=e$ and $p\circ \widetilde{\alpha}\simeq \alpha$,\ rel $\dot{I}$.
\end{definition}
\begin{theorem}\label{P2.13}
    An h-fibration has homotopically path lifting property.
  \begin{proof}
  Let $p:E\to B$ be an h-fibration, $\alpha$ be a path in $B$ and $e\in p^{-1}(\alpha(0))$.\ Also, let $F:{\{\ast}\}\times I\to B$ be the homotopy $F(\ast,t)=\alpha(t)$ and $\widetilde{f}:{\{\ast}\}\to E$ be the map $\widetilde{f}(\ast)=e$. Then $p\circ \widetilde{f}=F\circ J_0$ and since $p$ is an h-fibration, there is a homotopy $\widetilde{F}:{\{\ast}\}\times I\to E$ and a fiber homotopy $H:{\{\ast}\}\times I\to E$ such that $p\circ \widetilde{F}=F$ and $H:\widetilde{F}\circ J_0\simeq_p \widetilde{f}$.\ Let $\widetilde{\alpha}$ be the path in $E$ defined by $\widetilde{\alpha}(t)=\widetilde{F}(\ast,t)$.\ Then $H(\ast,0)=\widetilde{F}\circ J_0(\ast)=\widetilde{\alpha}(0)$, $H(\ast,1)=\widetilde{f}(\ast)=e$ and $$\ p\circ H(\ast,t)=p\circ \widetilde{F}\circ J_0(\ast)=p\circ \widetilde{f}(\ast)=\alpha(0).$$ Let $\widetilde{\gamma}:=H(\ast,-)$ which is a path in the fiber $p^{-1}(\alpha(0))$ from $\widetilde{\alpha}(0)$ to $e$.\ Then $\widetilde{\beta}=\widetilde{\gamma}^{-1}\ast \widetilde{\alpha}$ is a homotopical lift of $\alpha$ started from $e$, because $$p\circ \widetilde{\beta}=p\circ (\widetilde{\gamma}^{-1}\ast \widetilde{\alpha})=(p\circ \widetilde{\gamma}^{-1})\ast (p\circ \widetilde{\alpha})\simeq C_{\alpha(0)}\ast \alpha\simeq \alpha,\ rel\ \dot{I}.$$
  \end{proof}
\end{theorem}
We know that restriction of a fibration on each of whose path components is a fibration and for maps with locally path connected total space, we have the converse (see, \cite[Lemma 2.3.1 and Theorem 2.3.2]{SP}).\ These results are satisfied for h-fibrations with a simulated proof which is left to the readers.
  \begin{proposition}
    Let $p:E\to B$ be a map.\ If $E$ is locally path connected, then $p$ is an h-fibration if and only if for each path component $A$ of $E$, $p(A)$ is a path component of $B$ and $p|_A:A\to p(A)$ is an h-fibration.
  \end{proposition}
   Let $p:E\to B$ be a map and define a subspace $\overline{B}\subseteq E\times B^I$ as follows:
 $$\overline{B}={\{(e,\omega)\in E\times B^I|\omega(0)=p(e)}\}.$$
Recall that, a lifting function for $p$ is a map $\lambda:\overline{B}\to E^I$ which assigns to each point $e\in E$ and any path $\omega$ in $B$ starting at $p(e)$ a path $\lambda(e,\omega)$ in $E$ starting at $e$ that is a lift of $\omega$.  Existence of a lifting function for a map $p:E\to B$ is equivalent to $p$ is a fibration (see \cite[Theorem 2.7.8]{SP}). For h-fibrations we introduce a homotopical version of lifting function and show that every h-fibration has one of them.
\begin{definition}
 A \emph{homotopically lifting function for $p$} is a map $\lambda:\overline{B}\to E^I$
 which assigns to each point $e\in E$ and any path $\omega$ in $B$ starting at $p(e)$ a path $\lambda(e,\omega)$ in $E$ starting at $e$ that is a homotopical lift of $\omega$.
 \end{definition}

\begin{theorem}\label{p3.12}
Every h-fibration has a homotopically lifting function.
\begin{proof}
  Let $p:E\to B$ be an h-fibration.\ Define two maps $\wt f:\overline{B}\to E$ and $F:\overline{B}\times I\to B$ by $\wt f(e,\omega)=e$ and $F((e,\omega),t)=\omega(t)$, respectively.\ Since $F\circ J_0(e,\omega)=F((e,\omega),0)=\omega(0)=p(e)=p\circ \wt f(e,\omega)$ and $p$ is an h-fibration, there exist homotopies $\widetilde{F}, H:\overline{B}\times I\to E$ such that $p\circ \widetilde{F}=F$ and $H:\widetilde{F}\circ J_0\simeq_p \wt f$.\ Define $\lambda:\overline{B}\to E^I$ by $\lambda(e,\omega)(t)=\widetilde{F}((e,\omega),t)$ which is continuous.\ Let $\tilde{x}:=\lambda(e,\omega)(0)$.\ Then $\tilde{x}\in p^{-1}(p(e))$ because $p(\tilde{x})=p\circ \lambda(e,\omega)(0)=p\circ \widetilde{F}\circ J_0(e,\omega)=p\circ \widetilde{f}(e,\omega)=p(e)$.\ Similar to the proof of Proposition \ref{P2.13}, there is a path $\gamma$ in the fiber $p^{-1}(p(e))$ from $e$ to $\tilde{x}$.\ Define $\lambda':\overline{B}\to E^I$ by $\lambda'(e,\omega)=\gamma\ast\lambda(e,\omega)$.\ Then $\lambda'(e,\omega)(0)=\gamma(0)=e$ and $$p\circ \lambda'(e,\omega)=p\circ (\gamma\ast\lambda(e,\omega))=(p\circ \gamma)\ast(p\circ \lambda(e,\omega))=C_{p(e)}\ast p\circ \lambda(e,\omega)$$ $$=C_{p(e)}\ast p\circ \widetilde{F}((e,\omega),-)=C_{p(e)}\ast F((e,\omega),-)=C_{p(e)}\ast \omega\simeq \omega,\ \ \\rel\ \dot{I}.$$ Therefore $\lambda'$ is a homotopically lifting function for $p$.
\end{proof}
\end{theorem}
\begin{remark}\label{R3.16}
  The converse of Theorem \ref{p3.12} is not true.\ As an example, let $E=I\times I-{\{(0,\frac{1}{2})}\}$,\ $B=I$ and $p$ be the projection on the first component.\ Since the fibers of $p$ do not have the same homotopy, $p$ is not an h-fibration.\ However, $p$ has a homotopically lifting function. For, let $e\in E$ and $\omega$ be a path in $B$ starting at $p(e)$.\ Also define two paths $\alpha, \beta$ in $E$ by $\alpha(t)=(1-t)e+tA$ and $\beta(t)=(1-t)A+t(\omega(1),0)$, where $A=(1,\frac{1}{2})$.
  \ Define $\lambda:\overline{B}\to E^I$ such that $\lambda(e,\omega)(t)=(\alpha\ast\beta)(t)$.\ Then $\lambda(e,\omega)$ is a path starting at $\lambda(e,\omega)(0)=\alpha(0)=e$.\ Moreover, since $$p\circ \lambda(e,\omega)(0)=p\circ \alpha(0)=p(e)=\omega(0),$$ $$p\circ \lambda(e,\omega)(1)=p\circ \beta(1)=p(\omega(1),0)=\omega(1)$$ and $B$ is simply connected, $p\circ \lambda(e,\omega)(t)\simeq \omega(t)\ rel\ \dot{I}$, as desired.
\end{remark}


\section{Category of h-Fibrations}
In this section, $\mathrm{Fib}$ and $\mathrm{Fib(B)}$ are the category of fibrations and fibrations over $B$, and have the categories $\mathrm{Fibu}$ and $\mathrm{Fibu(B)}$ (with the extra assumption unique path lifting) as subcategory, respectively (see \cite{SP}).\ When we deal with fibrations with wuphl instead of upl, we have the categories $\mathrm{Fibwu}$ and $\mathrm{Fibwu(B)}$ \cite{MPT}, for which
 $$\mathrm{Fibu}\leq \mathrm{Fibwu},\ \ \ \mathrm{Fibu(B)}\leq \mathrm{Fibwu(B)}.$$

 To construct new categories by h-fibrations, we need to the following essential proposition.
 \begin{proposition}\label{P3.1}
   Composition of two h-fibrations is an h-fibration.
   \begin{proof}
 Let $p':E'\to E$ and $p:E\to B$ be two h-fibrations, $\widetilde{f}:X\to E'$ and $F:X\times I\to B$ be two maps such that $(p\circ p')\circ \widetilde{f}=F\circ J_0$.\ Since $p$ is an h-fibration, there exist homotopies $\widetilde{F}, H:X\times I\to E$ such that $p\circ \widetilde{F}=F$ and $H:p'\circ \widetilde{f}\simeq_p\widetilde{F}\circ J_0 $. Since $p'\circ \widetilde{f}=H\circ J_0$ and $p'$ is an h-fibration, there exist homotopies $\widetilde{H}, D:X\times I\to E'$ such that $p'\circ \widetilde{H}=H$ and $D:\widetilde{f}\simeq_{p'}\widetilde{H}\circ J_0$.\ Let $J_1:X\to X\times I$ be the map $J_1(x)=(x,1)$. Since $p'\circ (\widetilde{H}\circ J_1(x))=p'\circ \widetilde{H}(x,1)=H(x,1)=\widetilde{F}\circ J_0(x)$ and $p'$ is an h-fibration, there exist homotopies $\widetilde{K}, T:X\times I\to E'$ such that $p'\circ \widetilde{K}=\widetilde{F}$ and $T:\widetilde{H}\circ J_1\simeq_{p'}\widetilde{K}\circ J_0$. Since, $(p\circ p')\circ \widetilde{K}=p\circ \widetilde{F}=F$, $\wt K$ is the desired homotopy.  Also, by Proposition \ref{P2.2}, $D:\widetilde{f}\simeq_{p\circ p'}\widetilde{H}_0$ and $T:\widetilde{H}_1\simeq_{p\circ p'}\widetilde{K}_0$.
   Moreover,
 \\$(p\circ p')\circ \widetilde{H}(x,t)=p\circ H(x,t)=p\circ H(x,0)=(p\circ p')\circ \wt H(x,0)=(p\circ p')\circ \wt H_0(x)$,
  \\$(p\circ p')\circ \widetilde{H}(x,t)=p\circ H(x,t)=p\circ H(x,1)=(p\circ p')\circ \wt H(x,1)=(p\circ p')\circ \wt H_1(x)$, which imply that $\wt H:\widetilde{H}_0\simeq_{p\circ p'}\widetilde{H}_1$. Since fiber homotopy is an equivalence relation, $\widetilde{f}\simeq_{p\circ p'}\widetilde{K}\circ J_0$ and so the result holds.
   \end{proof}
  \end{proposition}
It is straightforward that composition of two maps with wuphl is a map with wuphl \cite[Proposition 4.1]{MPT} and so we have the following proposition.
 \begin{proposition}\label{P3.2}
   Composition of h-fibrations with wuphl is an h-fibration with wuphl.
   \end{proposition}
Now, we can define category of h-fibrations, $\mathrm{hFib}$ and its subcategory, category of h-fibrations with wuphl, $\mathrm{hFibwu}$ in which the objects are topological spaces and morphisms are h-fibrations and h-fibrations with wuphl, respectively.\ Moreover, for a given space $B$, we can consider other categories, $\mathrm{hFib(B)}$ and  $\mathrm{hFibwu(B)}$, whose objects are h-fibrations and h-fibrations with wuphl over $B$ and morphisms are the commutative triangles.

    By Theorem \ref{T2.10}, since every fibration is an h-fibration, we have the following diagram of inclusion relations between categories.\\ \\ \\ \\

\unitlength 1mm 
\linethickness{0.4pt}
\ifx\plotpoint\undefined\newsavebox{\plotpoint}\fi 
\begin{picture}(97.544,50.979)(0,0)
\put(36.474,50.979){Fibu}
\put(60.755,50.874){Fibwu}
\put(94.075,50.874){Fib}
\put(33.741,27.119){hFibu}
\put(60.334,27.014){hFibwu}
\put(93.34,27.224){hFib\ .}
\put(46.249,50.874){\vector(1,0){11.878}}
\put(46.249,27.224){\vector(1,0){11.878}}
\put(78.729,50.874){\vector(1,0){11.878}}
\put(78.729,27.224){\vector(1,0){11.878}}
\put(67.272,44.988){\vector(0,-1){11.142}}
\put(97.544,44.883){\vector(0,-1){11.142}}
\put(40.153,33.951){\vector(0,-1){.07}}\put(40.153,44.988){\vector(0,1){.07}}\put(40.153,44.988){\line(0,-1){11.0368}}
\end{picture}

 It is notable that we have a similar diagram for the categories constructed over the base space $B$.
Also, note that in the above diagram, the inclusions are proper.
The first row is proper \cite[Example 3.3]{MPT} and the second row is proper, since a fibration is an h-fibration.\ Moreover, Example \ref{ex1} shows that the second and the third column are proper.

  Now, we study the existence of products and coproducts for these categories.\
 \begin{proposition}Product of two h-fibrations is an h-fibration.\label{P3.4}
   \begin{proof}
   Let $p:E\to B$ and $p':E'\to B'$ be two h-fibrations, $\widetilde{f}:X\to E\times E'$ and $F:X\times I\to B\times B'$ be maps such that $(p\times p')\circ \widetilde{f}=F\circ J_0$. Since $p\circ pr_1\circ\widetilde{f}=(pr_{1}\circ F)\circ J_0$, $p'\circ pr_2\circ \widetilde{f}=(pr_{2}\circ F)\circ J_0$ and $p$ and $p'$ are h-fibration, there exist $\widetilde{F}_1:X\times I\to E$ and $\widetilde{F}_2:X\times I\to E'$ such that $p\circ \widetilde{F}_1=pr_1\circ F$, $p'\circ \widetilde{F}_2=pr_2\circ F$, $\widetilde{F}_1\circ J_0\simeq_p pr_1\circ \wt{f}$ and $\widetilde{F}_2\circ J_0\simeq_{p'} pr_2\circ \wt{f}$.\ Define $\widetilde{F}:X\times I\to E\times E'$ by $\widetilde{F}(x,t)=(\widetilde{F}_1(x,t), \widetilde{F}_2(x,t))$.\ Then
   $$(p\times p')\circ \widetilde{F}=(p\circ \widetilde{F}_1, p'\circ \widetilde{F}_2)=(pr_1\circ F,pr_2\circ F)=F,$$
    and
     $$\widetilde{F}\circ J_0=(\widetilde{F}_1\circ J_0, \widetilde{F}_2\circ J_0)\simeq_{p\times p'} (pr_1\circ \widetilde{f}, pr_2\circ \widetilde{f})=\widetilde{f}.$$
   \end{proof}
 \end{proposition}
 It is easy to see that product of two maps with wuphl is a map with wuphl. Hence we have the following result.
 \begin{proposition}
   The categories $\mathrm{hFib}$ and $\mathrm{hFibwu}$ have the product.
 \end{proposition}
 To present products for $\mathrm{hFib(B)}$ and $\mathrm{hFibwu(B)}$, consider the Whitney sum of h-fibrations (with wuphl).\ If ${\{p_{j}:E_{j}\to B|j\in J}\}$ is an indexed collection of h-fibrations (with wuphl ) over the space $B$, define $$\oplus_{B,J}E_j={\{(e_{j})_{j}\in \sqcap_j{E_j}|e_{j}\in E_{j},\ \mathrm{and}\ p_{j}(e_{j})=p_{i}(e_{i}),\ \mathrm{for}\ i,j\in J}\}$$
and also define $$\oplus_{B,J}p_{j}:\oplus_{B,J}E_j\to B$$ $$(e_{j})_{j}\rightarrowtail p_{j}(e_{j}).$$
\begin{proposition} Let ${\{p_{j}:E_{j}\to B|j\in J}\}$ be an indexed collection of h-fibrations (with wuphl) on the space $B$.\ Then $\oplus_{B,J}p_{j}$ is an h-fibration (with wuphl).
\begin{proof}
Let $E:=\oplus_{B,J}E_j$ and $p:=\oplus_{B,J}p_{j}$.\ Also, let $\widetilde{f}:X\to E$ and $F:X\times I\to B$ be two maps such that $p\circ \widetilde{f}=F\circ J_0$.\ Then $\widetilde{f}=(\widetilde{f}_j)_j,\ j\in J$, where $\widetilde{f}_j:X\to E_j$ is the projection of $\widetilde{f}$ over the j-th component.\ By definition of $p$, $p_j\circ \widetilde{f}_j=F\circ J_0$ and since $p_j$ is an h-fibration, there exist homotopies $\widetilde{F}_j:X\times I\to E_j$ and $H_j:X\times I\to E_j$ such that $p_j\circ \widetilde{F}_j=F$ and $H_j:\widetilde{F}_j\circ J_0\simeq_{p_j}\widetilde{f}_j$.\ Since $p_j\circ \widetilde{F}_j=F=p_i\circ \widetilde{F}_i$, we can define $\widetilde{F}:X\times I\to E$ by $\widetilde{F}(x,t)=(\widetilde{F}_j(x,t))_j$.\ Hence, $$p\circ \widetilde{F}=p\circ (\widetilde{F}_j)_j=p_j\circ \widetilde{F}_j=F.$$
Also, since $p_j\circ H_j(x,t)=p_j\circ \widetilde{f}_j(x)=p_i\circ \widetilde{f}_i(x)=p_i\circ H_i(x,t)$, $Im (H)\subseteq E$ and so we can define $H:X\times I\to E$ by $H(x,t)=({H}_j(x,t))_j$.\ Now $H:\widetilde{F}\circ J_0\simeq_p\widetilde{f}$ since
 \begin{align*}
   &  H(x,0)=(H_j(x,0))_j=(\widetilde{F}_j\circ J_0(x))_j=(\widetilde{F}_j)_j\circ J_0(x)=\widetilde{F}\circ J_0(x), \quad \\
   &  H(x,1)=(H_j(x,1))_j=(\widetilde{f}_j(x))_j=\widetilde{f}(x), \quad \\
   &  p\circ H(x,t)=p\circ (H_j(x,t))_j=p_j\circ H_j(x,t)=p_j\circ (\widetilde{F}_j\circ J_0)(x)=\quad \\
   & (p_j\circ \widetilde{F}_j)\circ J_0(x)=(p\circ \widetilde{F})\circ J_0(x)=p\circ (\widetilde{F}\circ J_0)(x),\quad\\
   &  p\circ H(x,t)=p\circ (H_j(x,t))_j=p_j\circ H_j(x,t)=p_j\circ \widetilde{f}_j(x)=p\circ \widetilde{f}(x).
\end{align*}
Therefore $p$ is an h-fibration.\ Moreover, if every $p_j$ has wuphl, since the fibers of $p$ are the product of the fibers of $p_j$'s, then by Proposition \ref{P2.6}, $p$ has wuphl.
\end{proof}
\end{proposition}
The following result is a consequence of the above proposition.
\begin{theorem}
The categories $\mathrm{hFib(B)}$ and $\mathrm{hFibwu(B)}$ have products.\
\end{theorem}
 Suppose ${\{p_{j}:E_{j}\to B_j|j\in J}\}$ is an indexed collection of morphisms in $\mathrm{hFib}$ (or $\mathrm{hFibwu}$), $E:=\sqcup_j E_j$ and $B:=\sqcup_j B_j$ are disjoint union of $E_j$'s and $B_j$'s, respectively.\ Define $q:E\lo B$ by $q|_{E_j}=p_j$.\ Then $q$ is an h-fibration (with wuphl).\ Because let $\widetilde{f}:X\to E$ and $F:X\times I\to B$ be the maps such that $q\circ \widetilde{f}=F\circ J_0$.\ If $x_0\in X$, then there exists one and only one $j\in J$ such that $\widetilde{f}(x_0)\in E_j$ and $F\circ J_0(x_0)\in B_j$.\ Since $E_j$'s and $B_j$'s are disjoint, continuity of $\widetilde{f}$ and $F\circ J_0$ yields that for every $x\in X$ and every $t\in I$, $\widetilde{f}(x)\in E_j$, $F\circ J_0(x)\in B_j$ and $F(x,t)\in B_j$ which imply that $p_j\circ \widetilde{f}(x)=F\circ J_0(x)$.\ By assumption, there exist homotopies $\widetilde{F}_j:X\times I\to E_j$ and $H_j:X\times I\to E_j$ such that $p_j\circ \widetilde{F}_j=F$ and $H_j:\widetilde{F}_j\circ J_0\simeq_{p_j}\widetilde{f}$.\ Define $\widetilde{F}, H:X\times I\to E$ by $\widetilde{F}(x,t)=\widetilde{F}_j(x,t)$ and $H(x,t)=H_j(x,t)$, if $\widetilde{f}(x)\in E_j$.\ Therefore $q\circ \widetilde{F}(x,t)=q\circ \widetilde{F}_j(x,t)=p_j\circ \widetilde{F}_j(x,t)=F(x,t)$.\ Also, $H:\widetilde{F}\circ J_0\simeq_q\widetilde{f}$ because if $\widetilde{f}(x)\in E_j$, then
 \begin{align*}
    & H(x,0)=H_j(x,0)=\widetilde{F}_j\circ J_0(x),  \quad  \\
    &  H(x,1)=H_j(x,1)=\widetilde{f}(x),\quad\\
    & q\circ H(x,t)=q\circ H_j(x,t)=p_j\circ H_j(x,t)=p_j\circ(\widetilde{F}_j\circ J_0)(x)=\quad\\
    & (p_j\circ \widetilde{F}_j)\circ J_0(x)=(q\circ \widetilde{F})\circ J_0(x)=q\circ (\wt F\circ J_0)(x),\quad \\
   & q\circ H(x,t)=q\circ H_j(x,t)=p_j\circ H_j(x,t)=p_j\circ \widetilde{f}(x)=q\circ \widetilde{f}(x).
 \end{align*}
 Hence $q$ is an h-fibration.\ Moreover, if every $p_j$ has wuphl, since a fiber of $q$ is a fiber of one of ${p_j}^{,}s$, Proposition \ref{P2.6} follows that $q$ has wuphl.
 \\ Similarly, if ${\{p_{j}:E_{j}\to B|j\in J}\}$ is an indexed collection of objects in $\mathrm{hFib(B)}$ (or $\mathrm{hFibwu(B)}$), $q:\sqcup_{j} E_j\lo B$ defined by $q|_{E_j}=p_j$ is also an h-fibration (with wuphl), because it is sufficient that for every $j\in J$, let $B_j:=B$.\ Therefore, we have the following result.\
\begin{theorem}\label{T4.7}
The categories  $\mathrm{hFib}$, $\mathrm{hFibwu}$, $\mathrm{hFib(B)}$ and $\mathrm{hFibwu(B)}$ have coproducts.\
\end{theorem}\ \\
\textbf{References}

\end{document}